\documentclass[11pt,final]{iopart}
\usepackage{amssymb}
\usepackage{amsfonts}
\usepackage{amsthm}
\usepackage{iopams}
\usepackage{tikz}
\usepackage{tikz-cd}
\usepackage{etex}        
\usepackage[dvips]{epsfig}
\usepackage{float}
\usepackage{footmisc}
\usepackage{showkeys}
\usepackage{todo}
\usepackage{prettyref}
\usepackage{verbatim}
\usepackage[latin1]{inputenc}
\usepackage{tikz}
\usepackage{subfigure}
\usepackage[all]{xy}
\usepackage{verbatim}
\usetikzlibrary{shapes,arrows}
\newrefformat{eq}{(\ref{#1})}
\newrefformat{fig}{Figure~\ref{#1}}
\usepackage[pdfpagelabels]{hyperref}
\hypersetup{colorlinks=true,
linkcolor=blue,
citecolor=blue,
plainpages=false}
\newtheorem{theorem}{\sc Theorem}[section]
\newtheorem{propn}[theorem]{\sc Proposition}

\newtheorem{definition}[theorem]{\sc Definition}

\newtheorem{cor}[theorem]{\sc Corollary}

\newcommand{\argmin}{{\rm arg\,min}}

\eqnobysec

\makeindex

\newcommand{\cB}{{\mathcal B}}
\newcommand{\cC}{{\mathcal C}}
\newcommand{\cD}{{\mathcal D}}

\newcommand{\cH}{{\mathcal H}}
\newcommand{\cI}{{\mathcal I}}

\newcommand{\cK}{{\mathcal K}}
\newcommand{\cL}{{\mathcal L}}

\newcommand{\cO}{{\mathcal O}}
\newcommand{\cP}{{\mathcal P}}

\newcommand{\cR}{{\mathcal R}}

\newcommand{\cT}{{\mathcal T}}
\newcommand{\cU}{{\mathcal U}}

\newcommand{\cW}{{\mathcal W}}

\newcommand{\tx}{\tilde{x}}
\newcommand{\ty}{\tilde{y}}
\newcommand{\tz}{\tilde{z}}

\newcommand{\bx}{{\bf x}}

\newcommand{\R}{\mathbb{R}}

\newcommand{\norm}[1]{|| #1||}

\newcommand{\myarr}[1]{\left( \begin{array}{cccccccccccccccccc}#1\end{array}\right)}

\newcommand{\ba}{\begin{array}}
\newcommand{\ea}{\end{array}}
\newcommand{\be}{\begin{equation}}
\newcommand{\ee}{\end{equation}}
\newcommand{\bea}{\begin{eqnarray}}
\newcommand{\eea}{\end{eqnarray}}
\newcommand{\beq}{\begin{equation}}
\newcommand{\eeq}{\end{equation}}
\newcommand{\bqt}{\begin{quote}}
\newcommand{\eqt}{\end{quote}}

\begin{document}

%
\title[Regularity Properties of the Smoothed-TV]
{Smoothed-TV Regularization for H\"{o}lder Continuous Functions}

%
\author{Erdem Altuntac}

\address{Institute for Numerical and Applied Mathematics,
University of G\"{o}ttingen, Lotzestr. 16-18,
D-37083, G\"{o}ttingen, Germany}


\ead{\mailto{e.altuntac@math.uni-goettingen.de}}

\begin{abstract}

This work aims to explore the regularity properties of 
the smoothed-TV regularization for the functions is of
the class H\"{o}lder continuous. 
Over some compact and convex domain $\Omega,$
we study construction of multivariate function
$\varphi(\bx) :\Omega \subset \R^{3} \rightarrow \R_{+}$
as the optimized solution to the following convex 
minimization problem

\bea
\argmin_{\Omega} \left\{
F_{\alpha}(\cdot , f^{\delta}) := \frac{1}{2} \norm{\cT( \cdot ) - f^{\delta}}_{\cH}^2 + \alpha J(\cdot) \right\} ,
\nonumber
\eea
where the penalizer 
$J(\cdot) : \cC^{1}(\Omega,\R^{3})\rightarrow \R_{+}$ is 
the smoothed total variation penalizer

\bea
J(\cdot) = \int_{\Omega} \sqrt{\norm{\nabla(\cdot)}_2^2 + \beta} d \bx ,
\nonumber
\eea

\noindent for a fixed $0 < \beta < 1.$ We assume our target function to be H\"{o}lder
continuous. With this assumption, we establish relation between total 
variation of our target function and its H\"{o}lder coefficient.
We prove that the smoothed-TV regularization is an admissible
regularization strategy by evaluating the discrepancy 
$\norm{\cT\varphi_{\alpha} - f^{\delta}} \leq \tau\delta$
for some fixed $\tau \geq 1.$
To do so, we need to assume that the target function
to be class of $\cC^{1+}(\Omega).$ From here, under the fact that the penalty 
$J(\cdot)$ is strongly convex, we move on
to showing the convergence of $\norm{\varphi_{\alpha} - \varphi^{\dagger}},$
for $\varphi_{\alpha}$ is the optimum and $\varphi^{\dagger}$
is the true solution for the given minimization problem above. 
We demonstrate that strong convexity
and $2-$convexity are actually different names for the same
concept. In addition to these facts,
we make us of Bregman divergence in order to be able to quantify
the rate of convergence.

\bigskip
\textbf{Keywords.}
{H\"{o}lder continuity, Bounded variation, smoothed total variation, Morozov discrepancy.}
\end{abstract}

\bigskip


\section{Introduction}

As alternative to well established Tikhonov regularization,
\cite{Tikhonov63, TikhonovArsenin77},
studying convex variational regularization with any penalizer $J(\cdot)$
has become important over the last decade. Introducing a new image 
denoising method named as {\em total variation}, \cite{RudinOsherFatemi92},
is commencement of this study.
Application and analysis of the method have been widely carried
out in the communities of inverse problems and optimization, 
\cite{AcarVogel94, BachmayrBurger09, BardsleyLuttman09, ChambolleLions97,
ChanChen06, ChanGolubMulet99, DobsonScherzer96, 
DobsonVogel97, VogelOman96}. Particularly, formulating
the minimization problem as variational problem and
estimating convergence rates with variational source conditions
has also become popular recently, \cite{BurgerOsher04, Grasmair10, 
Grasmair13, GrasmairHaltmeierScherzer11, Lorenz08}.
Unlike in the available literature, 
we define discrepancy principle for the smoothed-TV regularization
under a particular rule for the choice of regularization parameter.
Furthermore, still with the same regularization parameter,
we manage to show that smoothed-TV regularization is an 
admissible regularization strategy with H\"{o}lder continuity.

We are tasked with constructing the regularized solution
$\varphi_{\alpha}$ over some compact and convex domain 
$\Omega \subset \cH,$ for the following variational minimization 
problem,

\beq
\label{problem}
\varphi_{\alpha(\delta)} \in \argmin_{\varphi \in \cH} 
\left\{ F_{\alpha}(\varphi , f^{\delta}) := 
\frac{1}{2} \norm{\cT \varphi - f^{\delta}}_{\cH}^2 + \alpha J(\varphi) \right\} ,
\eeq
for the penalty term $J(\varphi) : \cC^{2}(\Omega , \cH) \rightarrow \R_{+}$ defined by

\bea
\label{smoothed_tv_regul}
J(\varphi) = \int_{\Omega} \sqrt{\norm{\nabla \varphi}_2^2 + \beta} d \bx ,
\eea
and $\alpha > 0$ is the regularization parameter.
It is expected that the perturbed given data is 
$f^{\delta} \notin \cR(\cT)$ lies in in some 
$\delta-$ball $\cB_{\delta}(f^{\dagger})$
centered at the true data $f^{\dagger},$
{\em i.e.} $\norm{f^{\dagger} - f^{\delta}} \leq \delta.$
The compact forward operator $\cT : \Omega \subset \cH \rightarrow \cH$
is assumed to be linear and injective.
It is well known by the theory of inverse problems 
that a regularization strategy is admissible
if the regularization parameter satisfies,

\beq
\label{discrepancy_pr_definition}
\alpha(\delta , f^{\delta}) = \sup \{ \alpha > 0 \mbox{ }\vert \mbox{ }\norm{\cT\varphi_{\alpha(\delta)} - f^{\delta}}\leq \tau\delta \} ,
\eeq
where $\tau \geq 1,$ \cite[Eq. (4.57) and (4.58)]{Engl96},
\cite[Definition 2.3]{Kirsch11}. The regularized solution 
$\varphi_{\alpha(\delta)}$ of the problem (\ref{problem}) 
must satisfy the following first order optimality conditions,

\bea
\label{optimality_1}
& 0 & = \nabla F_{\alpha}(\varphi_{\alpha(\delta)})
\nonumber\\
& 0 & = \cT^{\ast}(\cT\varphi_{\alpha(\delta)} - f^{\delta}) + \alpha(\delta) \nabla J(\varphi_{\alpha(\delta)})
\nonumber\\
& \cT^{\ast}(f^{\delta} - \cT\varphi_{\alpha(\delta)}) & = \alpha(\delta) \nabla  J(\varphi_{\alpha(\delta)}) .
\eea

This work aims to answer two fundamental questions 
in the field of regularization theory; Is it possible to quantify $\tau$ in 
(\ref{discrepancy_pr_definition}) when the penalizer is
(\ref{smoothed_tv_regul})? What is the rule for the choice
of regularization parameter $\alpha(\delta , f^{\delta})$
when the penalizer is (\ref{smoothed_tv_regul}) that the
smoothed-TV is also an admissible regularization theory?
We will be able to quantify the rate of the convergence of 
$\norm{\varphi_{\alpha(\delta)} - \varphi^{\dagger}}$
by means of the Bregman divergence.

Existence of the solution to the TV minimization problem, {\em i.e.} 
$J(\cdot) = \int_{\Omega} \norm{\nabla(\cdot)}_2 dx$ 
in the problem (\ref{problem}), has been discussed extensively 
\cite{Hintermuller14, Setzer11}. Moreover, an existence and uniquness 
theorem for the minimizer of quadratic functionals 
with different type of convex integrands has been established 
in \cite[Theorem 9.5-2]{Ciarlet13}. As has been given by 
the {\em Minimal Hypersurfaces} problem in \cite{Ekeland74}, the minimizer
of the problem (\ref{problem}), for the smoothed-TV penalty
$J(\cdot) = \int_{\Omega} \sqrt{\norm{\nabla(\cdot)}_2^2 + \beta} dx,$
exists on a reflexive Banach space.

\section{Notations and Prerequisite Knowledge}
\label{notations}

\subsection{Vector calculus notations}
\label{vector_notations}

We assume to be tasked with reconstruction of some non-negative scalar
function defined on a compact subset $\Omega$ of 
$\R^{3},$ {\em i.e.} $\varphi(\bx) :\Omega \subset \R^{3} \rightarrow \R_{+}$
where the spatial coordinate is $\bx = (x, y, z).$  
Then the gradient of $\varphi$ is regarded as a vector with components

\begin{displaymath}
\nabla \varphi = \myarr{\frac{\partial\varphi}{\partial x}, & \frac{\partial\varphi}{\partial y}, & \frac{\partial\varphi}{\partial z}}^{T} .
\end{displaymath}
The magnitude of this gradient in the Euclidean sense,

\beq
\label{magnitude_of_the_grad}
\norm{\nabla\varphi}_2 = \left( \left\vert\frac{\partial\varphi}{\partial x}\right\vert^2 + 
\left\vert\frac{\partial\varphi}{\partial y}\right\vert^2 + 
\left\vert\frac{\partial\varphi}{\partial z}\right\vert^2 \right)^{1/2} .
\eeq

\subsection{Functional analysis notations}
\label{functional_notations}

We aim to approximate a function 
which belongs to H\"{o}lder space. H\"{o}lder space is denoted 
by $\cC^{0,\gamma}(\Omega)$ where $0 < \gamma \leq 1,$ \cite[Subsection 5.1]{Evans98}.
If a multivariate function $\varphi(\bx) \in \cC^{0,\gamma}(\Omega),$
then there exists $\kappa > 0$ such that the function $\varphi(\bx)$ 
satisfies the following H\"{o}lder continuity
\beq
\label{holder_cont}
\vert \varphi(\bx) - \varphi(\tilde{\bx}) \vert \leq \kappa \norm{ \bx - \tilde{\bx} }_2^{\gamma}, \mbox{ } \forall \bx, \tilde{\bx} \in \Omega .
\eeq
Here $\vert \cdot \vert$ is the absolute value of 
$\varphi(\bx) :\Omega \subset \R^{3} \rightarrow \R_{+}.$
H\"{o}lder space is a Banach space endowed with the norm

\beq
\label{holder_norm}
\norm{\varphi}_{\gamma} := \norm{\varphi}_{\infty} + [\varphi]_{\cC^{0,\gamma}} ,
\eeq
where the H\"{o}lder coefficient $[\varphi]_{\cC^{0,\gamma}(\Omega)}$ is defined by

\beq
\label{holder_coeff}
[\varphi]_{\cC^{0,\gamma}(\Omega)} := \sup_{\bx , \tilde{\bx} \in \Omega \subset \R^{3}} 
\frac{\vert \varphi(\bx) - \varphi(\tilde{\bx}) \vert}
{\norm{ \bx - \tilde{\bx} }_2^{\gamma}} ,
\eeq
and the Euclidean norm is

\beq
\label{holder_euclidean}
\norm{\bx - \tilde{\bx}}_2^{\gamma} := \left( (x - \tx_{o})^2 + (y - \ty_{o})^2 + (z - \tz_{o})^2 \right)^{\gamma/2}.
\eeq
So that, we define H\"{o}lder space by

\begin{displaymath}
\cC^{0 , \gamma}(\Omega) := \{ \varphi \in \cC(\Omega) : \norm{\varphi}_{\gamma} < \infty \} .
\end{displaymath}

In this work, we focus on total variation (TV) of a
function, \cite{ChambolleLions97, RudinOsherFatemi92}. 
With (\ref{magnitude_of_the_grad}), 
TV of our multivariate function is explicitly,

\bea
\label{tv_integral_form}
TV(\varphi) := \int_{\Omega} \norm{\nabla\varphi}_2 dx =
\int_{\Omega_{x}}\int_{\Omega_{y}}\int_{\Omega_{z}} 
\left( \left(\frac{\partial\varphi}{\partial x}\right)^2 + 
\left(\frac{\partial\varphi}{\partial y}\right)^2 + 
\left(\frac{\partial\varphi}{\partial z}\right)^2 \right)^{1/2} dx dy dz.
\nonumber
\eea
Total variation type regularization targets 
the reconstruction of bounded variation (BV) class of functions, 
\cite{Vogel02},

\beq
\label{bv_def}
\norm{\varphi}_{BV} := \norm{\varphi}_{\cL^1} + TV(\varphi).
\eeq


\subsection{Bregman divergence}
\label{bregman_divergence_def}

Following formulation emphasizes the functionality of the Bregman
divergence in proving the norm convergence of the minimizer
of the convex minimization problem to the true solution.

\begin{definition}[Total convexity and Bregman divergence]\cite[Definition 1]{Bredies09}
\label{total_convexity}

Let $\Phi : \cH \rightarrow \R \cup \{ \infty \}$ be a smooth and convex functional. 
Then $\Phi$ is called totally convex in $u^{\ast} \in \cH,$ 
if, for $\nabla \Phi(u^{\ast})$ and $\{ u \},$ it holds that

\bea
D_{\Phi}(u , u^{\ast}) =
\Phi(u) - \Phi(u^{\ast}) - \langle \nabla \Phi(u^{\ast}) , u - u^{\ast} \rangle \rightarrow 0
\Rightarrow \norm{u - u^{\ast}}_{\cH} \rightarrow 0
\nonumber
\eea
where $D_{\Phi}(u , u^{\ast})$ represents the {\em Bregman divergence}.

It is said that $\Phi$ is {\em q-convex} in $u^{\ast} \in \cH$
with a $q \in [2, \infty ),$ if for all $M > 0$ there exists a $c^{\ast} > 0$ 
such that for all $\norm{u - u^{\ast}}_{\cH} \leq M$ we have

\beq
\label{q_convexity}
D_{\Phi}(u , u^{\ast}) = 
\Phi(u) - \Phi(u^{\ast}) - \langle \nabla \Phi(u^{\ast}), u - u^{\ast} \rangle \geq c^{\ast} \norm{u - u^{\ast}}_{\cH}^q .
\eeq
\end{definition}

Throughout our norm convergence estimations, we refer to this
definition for the case of $2-$convexity.

In fact, another similar estimation 
to (\ref{q_convexity}), for $q = 2,$ can also be derived
by making further assumption about the functional $\Phi$
one of which is strong convexity with modulus $c,$
\cite[Definition 10.5]{BauschkeCombettes11}.
Below is this alternative way of obtaining (\ref{q_convexity})
when $q = 2.$

\begin{propn}
\label{proposition_q-convexity}
Let $\Phi : \cH \rightarrow \R \cup \{ \infty \}$ 
be $\Phi \in \cC^{2}(\cH)$ is strongly convex with 
modulus of convexity $c > 0,$ {\em i.e.} $\nabla^2 \Phi \succ cI,$ then

\bea
\label{lower_bound_for_bregman}
D_{\Phi}(u , v) > c \norm{u - v}^2 + \cO(\norm{u - v}^2) .
\eea
\end{propn}

\begin{proof}

Let us begin with considering the Taylor expansion of $\Phi,$
\beq
\Phi(u) = \Phi(v) + \langle \nabla \Phi(v) , u - v \rangle + 
\frac{1}{2} \langle \nabla^2 \Phi(v)(u - v) , u - v \rangle + \cO(\norm{u - v}^2) .
\eeq
Then the Bregman divergence 

\bea
D_{\Phi}(u , v) & = & 
\Phi(u) - \Phi(v) - \langle \nabla \Phi(v) , u - v \rangle
\nonumber\\
& = & \langle \nabla \Phi(v) , u - v \rangle + 
\frac{1}{2} \langle \nabla^2 \Phi(v)(u - v) , u - v \rangle + \cO(\norm{u - v}^2) 
- \langle \nabla \Phi(v) , u - v \rangle
\nonumber\\
& = & \frac{1}{2} \langle \nabla^2 \Phi(v)(u - v) , 
u - v \rangle + \cO(\norm{u - v}^2) .
\nonumber
\eea
Since $\Phi(\cdot)$ is striclty convex, due to strong convexity 
and $\Phi \in \cC^{2}(\cH),$ hence one obtains that

\bea
D_{\Phi}(u , v) > c \norm{u - v}^2 + \cO(\norm{u - v}^2) ,
\eea
where $c$ is the modulus of convexity.

\end{proof}


\subsection{Further Results on the H\"{o}lder Continuity}
\label{holder_relations}

We already have reviewed in Subsection \ref{functional_notations}
that the H\"{o}lder space $\cC^{0, \gamma}$ is a Banach space 
endowed with the norm, for all $x \neq y \in \Omega$
and $\Omega$ is a compact domain,

\beq
\label{holder_norm}
\norm{\varphi}_{\gamma} := \sup_{\bx \in \Omega} \vert \varphi(\bx) \vert + 
[\varphi]_{\cC^{0,\gamma}(\Omega)}
= \norm{\varphi}_{\infty} + [\varphi]_{\cC^{0,\gamma}(\Omega)} .
\eeq
Here the H\"{o}lder coefficient is obviously bounded by

\beq
\label{holder_coefficient}
[\varphi]_{\cC^{0,\gamma}(\Omega)} :=
\sup_{\bx , \tilde{\bx} \in \Omega} \frac{\vert \varphi(\bx) - \varphi(\tilde{\bx}) \vert}
{\norm{ \bx - \tilde{\bx} }_2^{\gamma}} \leq \kappa .
\eeq
Furthermore, following from (\ref{holder_norm}), 
an immediate conclusion can be formulated as follows.

\begin{propn}
\label{holder_embedded_into_L1}
Over the compact domain $\Omega,$
if $\varphi \in \cC^{0, \gamma}(\Omega), $ 
then $\varphi \in \cL^1(\Omega).$
\end{propn}

\begin{proof}
Since $\norm{\varphi}_{\infty} \geq \frac{1}{\vert \Omega \vert} \norm{\varphi}_{\cL^1}$
and $[\varphi]_{\cC^{0,\gamma}(\Omega)} > 0,$ then

\beq
\norm{\varphi}_{\gamma} \geq \norm{\varphi}_{\infty} \geq \frac{1}{\vert \Omega \vert} \norm{\varphi}_{\cL^1}.
\eeq

\end{proof}


\section{H\"{o}lder Continuity and TV of a $\cC^{1}-$Smooth Function}
\label{holder_tv_spaces}

We now come to the point where we start establishing the relations
between $\gamma-$H\"{o}lder continuity and TV of a function $\varphi$
on $\R^{3}.$ The following theorems will also serve us for determining
an implementable and unique regularization parameter appeared in 
the minimization problem (\ref{problem}).
We emphasize a very important assumption that we always work with 
continuous function on a compact domain which is uniformly
continuous. This fact will allow us to interchange the necessary operations
in order to obtain the desired results in what follows.
%

\begin{theorem}[Morrey's inequality]\cite[Subsection 5.6.2., Theorem 4]{Evans98}
\label{Morrey_ineq}
Let $\Omega \subset \R^{N}$ be the compact domain and let
$N < p \leq \infty.$ Then there exists a constant $C,$
depending only on $p$ and $N,$ such that

\beq
\norm{\varphi}_{\cC^{0,\gamma}(\Omega)} \leq C \norm{\varphi}_{\cW^{1,p}(\Omega)}
\eeq
for all $\varphi \in \cC^{1}(\Omega),$ where

\beq
\gamma := 1 - N/p .
\eeq
\end{theorem}

\begin{cor}
Specifically in $\R^{3},$ the theorem implies that

\beq
[\varphi]_{\cC^{0,1/4}(\Omega)} \leq C \norm{\nabla \varphi}_{\cL^4(\Omega)} ,
\eeq
since $\gamma := 1 - 3/4.$
\end{cor}

\begin{theorem}
\label{Holder_bounded_by_TV}
Over the compact domain $\Omega \subset \R^{3},$
with its volume $\vert \Omega \vert \in \R_{+},$ 
let $\varphi \in \cC^{1}(\Omega) \cap \cC^{0,1/4}(\Omega).$ 
Then H\"{o}lder coefficient $[\varphi]_{\cC^{0, 1/4}(\Omega)}$ 
of the function is bounded by its total variation $TV(\varphi)$ as such,

\begin{displaymath}
[\varphi]_{\cC^{0, 1/4}(\Omega)} \leq \frac{r(\Omega)^{3/4}}{\vert \Omega \vert} TV(\varphi), 
\mbox{ where } r(\Omega) := \sup_{\bx , \tilde{\bx} \in \Omega}\norm{\bx - \tilde{\bx}} .
\end{displaymath}
\end{theorem}

\begin{proof}
Recall our vectoral notations in $\R^{3},$
$\bx = (x, y, z)$ and $\tilde{\bx} = (\tx, \ty, \tz).$
Then for a fixed $\gamma \in (0,1],$ 
componentwise H\"{o}lder continuity in $\R^{3}$ is given by

\bea
[\varphi]_{\cC^{0,\gamma}(\Omega)} = 
\sup_{\bx , \tilde{\bx} \in \Omega \subset \R^{3}} \Bigg\{ \frac{\vert\varphi(x, y, z) - \varphi(\tx_{o}, \ty, \tz )\vert}{\norm{ \bx - \tilde{\bx} }_2^{\gamma} }
& , & \frac{\vert \varphi(x, y, z) - \varphi(\tx, \ty_{o}, \tz )\vert}{\norm{ \bx - \tilde{\bx} }_2^{\gamma}} , \cdots 
\nonumber\\
& \cdots & , \frac{\vert \varphi(x, y, z) - \varphi(\tx, \ty, \tz_{o} )\vert}{\norm{ \bx - \tilde{\bx} }_2^{\gamma}} \Bigg\} .
\nonumber
\eea
By the definition of Euclidean norm in (\ref{holder_euclidean}),

\begin{displaymath}
\norm{\bx - \tilde{\bx}}_2^{\gamma} \geq \sup_{\Omega}\{ \vert x - \tx_{o} \vert^{\gamma} , \vert y - \ty_{o} \vert^{\gamma} , \vert z - \tz_{o} \vert^{\gamma}\}.
\end{displaymath}
So this implies

\bea
[\varphi]_{\cC^{0,\gamma}(\Omega)} \leq 
\sup_{\bx , \tilde{\bx} \in \Omega \subset \R^{3}} \Bigg\{ \frac{\vert\varphi(x, y, z) - \varphi(\tx_{o}, \ty, \tz )\vert}{\vert x - \tx_{o} \vert^{\gamma} }
& , & \frac{\vert \varphi(x, y, z) - \varphi(\tx, \ty_{o}, \tz )\vert}{\vert y - \ty_{o} \vert^{\gamma}} ,
\nonumber\\
& , & \frac{\vert \varphi(x, y, z) - \varphi(\tx, \ty, \tz_{o} )\vert}{\vert z - \tz_{o} \vert^{\gamma}} \Bigg\} ,
\nonumber
\eea
\bea
 = \sup_{\bx , \tilde{\bx} \in \Omega \subset \R^{3}} \Bigg\{ \frac{\vert\varphi(x, y, z) - \varphi(\tx_{o}, \ty, \tz )\vert}{ \vert x - \tx_{o} \vert \vert x - \tx_{o} \vert^{\gamma - 1} }
& , &\frac{\vert \varphi(x, y, z) - \varphi(\tx, \ty_{o}, \tz )\vert}{\vert y - \ty_{o} \vert \vert y - \ty_{o} \vert^{\gamma - 1}} ,
\nonumber\\
& , & \frac{\vert \varphi(x, y, z) - \varphi(\tx, \ty, \tz_{o} )\vert}{\vert z - \tz_{o} \vert \vert z - \tz_{o} \vert^{\gamma - 1}} \Bigg\} ,
\nonumber
\eea
\bea
 = \sup_{\bx , \tilde{\bx} \in \Omega \subset \R^{3}} \Bigg\{ \frac{\vert\varphi(x, y, z) - \varphi(\tx_{o}, \ty, \tz )\vert}{ \vert x - \tx_{o} \vert }\vert x - \tx_{o} \vert^{1 - \gamma}
& , &\frac{\vert \varphi(x, y, z) - \varphi(\tx, \ty_{o}, \tz )\vert}{\vert y - \ty_{o} \vert }\vert y - \ty_{o} \vert^{1 - \gamma} , 
\nonumber\\
& , & \frac{\vert \varphi(x, y, z) - \varphi(\tx, \ty, \tz_{o} )\vert}{\vert z - \tz_{o} \vert }\vert z - \tz_{o} \vert^{1 - \gamma} \Bigg\} .
\nonumber
\eea
Here, the last equality in the chain is rather convenient to present since $\gamma - 1 < 0 < 1.$
Obviously, for any pair of points $(\bx , \tilde{\bx}) \in \Omega,$
there exists $s > 0$ such that s = $\norm{\bx - \tilde{\bx}}_2.$
Then,

\begin{displaymath}
r^{1 - \gamma} \geq s^{1 - \gamma} = \norm{\bx - \tilde{\bx}}_2^{1 - \gamma} \geq 
\sup_{\Omega}\{ \vert x - \tx_{o} \vert^{1 - \gamma} , \vert y - \ty_{o} \vert^{1 - \gamma} , \vert z - \tz_{o} \vert^{1 - \gamma}\},
\end{displaymath}
we have

\bea
[\varphi]_{\cC^{0,\gamma}(\Omega)} \leq 
\sup_{\bx , \tilde{\bx} \in \Omega \subset \R^{3}} \Bigg\{ & \frac{\vert\varphi(x, y, z) - \varphi(\tx_{o}, \ty, \tz )\vert}{ \vert x - \tx_{o} \vert }\norm{\bx - \tilde{\bx}}_2^{1 - \gamma} &
\nonumber\\ 
& , \frac{ \vert \varphi(x, y, z) - \varphi(\tx, \ty_{o}, \tz )\vert }{ \vert y - \ty_{o} \vert }\norm{\bx - \tilde{\bx}}_2^{1 - \gamma} &  
\nonumber\\
& , \frac{\vert \varphi(x, y, z) - \varphi(\tx, \ty, \tz_{o} )\vert}{\vert z - \tz_{o} \vert }\norm{\bx - \tilde{\bx}}_2^{1 - \gamma} \Bigg\}&
\nonumber
\eea
\bea
= \sup_{\bx , \tilde{\bx} \in \Omega \subset \R^{3} } \Bigg\{ \frac{\vert\varphi(x, y, z) - \varphi(\tx_{o}, \ty, \tz )\vert}{ \vert x - \tx_{o} \vert }s^{1 - \gamma} 
& , & \frac{ \vert \varphi(x, y, z) - \varphi(\tx, \ty_{o}, \tz )\vert }{ \vert y - \ty_{o} \vert }s^{1 - \gamma} , 
\nonumber\\
& , & \frac{\vert \varphi(x, y, z) - \varphi(\tx, \ty, \tz_{o} )\vert}{\vert z - \tz_{o} \vert }s^{1 - \gamma} \Bigg\} ,
\nonumber
\eea
\bea
\leq r^{1 - \gamma} \sup_{\bx , \tilde{\bx} \in \Omega \subset \R^{3} } \Bigg\{ \frac{\vert\varphi(x, y, z) - \varphi(\tx_{o}, \ty, \tz )\vert}{ \vert x - \tx_{o} \vert } 
& , & \frac{ \vert \varphi(x, y, z) - \varphi(\tx, \ty_{o}, \tz )\vert }{ \vert y - \ty_{o} \vert } ,
\nonumber\\
& , & \frac{\vert \varphi(x, y, z) - \varphi(\tx, \ty, \tz_{o} )\vert}{\vert z - \tz_{o} \vert } \Bigg\} .
\nonumber
\eea
Recall that our function $\varphi$ is continuous 
over the compact domain $\Omega$ which makes it
uniformly contiuous on the same domain.
Then we are allowed to interchage $\lim$ with $\sup.$
Now, moving on to the limit on both sides with respect to each component 
$\lim_{x \rightarrow \tx_0},$ $\lim_{y \rightarrow \ty_0}$ 
and $\lim_{z \rightarrow \tz_0},$ 

\bea
[\varphi]_{\cC^{0,\gamma}(\Omega)} & \leq & r^{1 - \gamma} \myarr{\left\vert \frac{\partial \varphi}{\partial x} \right\vert, 
\left\vert \frac{\partial \varphi}{\partial y} \right\vert, 
\left\vert \frac{\partial \varphi}{\partial z} \right\vert} 
\nonumber\\
& = & r^{1 - \gamma} \myarr{\left( \left\vert \frac{\partial \varphi}{\partial x} \right\vert^2 \right)^{1/2}, 
\left( \left\vert \frac{\partial \varphi}{\partial y} \right\vert^2 \right)^{1/2} , 
\left( \left\vert \frac{\partial \varphi}{\partial z} \right\vert^2 \right)^{1/2} } 
\nonumber\\
& \leq & r^{1 - \gamma} \norm{\nabla\varphi}_2
\nonumber
\eea
Again, the last inequality has been obtained by the fact
that sum of the components always remains greater than
each component itself. Now, integrate both sides over 
the compact domain $\Omega$ to yield

\begin{displaymath}
[\varphi]_{\cC^{0,\gamma}(\Omega)} \leq \frac{r^{1 - \gamma}}{\vert \Omega \vert} TV(\varphi) ,
\end{displaymath}
which is, to be more precise,

\begin{displaymath}
[\varphi]_{\cC^{0, 1/4}(\Omega)} \leq \frac{r^{3 / 4}}{\vert \Omega \vert} TV(\varphi) ,
\end{displaymath}
since $\gamma = 1/4$ in $\R^3.$

\end{proof}

This shows that H\"{o}lder coefficient of a function 
$\varphi \in \cC^{1}(\Omega) \cap \cC^{0,1/4}(\Omega)$
is an approximation for the total variation of the
same function. In the following theorems, we will
establish the reverse direction of this statement.
To do so, we will make use of the Lipschitz continuity
which is a specific case of H\"{o}lder continuity 
in (\ref{holder_cont}) for $\gamma = 1.$

\begin{theorem}
\label{holder_embedded_into_TV}

Under the same conditions of Theorem \ref{Holder_bounded_by_TV}, 
for $\gamma = 1$ in (\ref{holder_cont}),

\beq
\norm{\nabla \varphi}_{\cL^{1}} \leq \kappa \vert \Omega \vert .
\eeq
\end{theorem}

\begin{proof}

As we have introduced in the Section \ref{notations} by (\ref{magnitude_of_the_grad}),

\beq
\label{gradient_norm_bound}
\norm{\nabla\varphi}_2 = \left( \left\vert\frac{\partial\varphi}{\partial x}\right\vert^2 + 
\left\vert\frac{\partial\varphi}{\partial y}\right\vert^2 + 
\left\vert\frac{\partial\varphi}{\partial z}\right\vert^2 \right)^{1/2} 
\leq \left\vert\frac{\partial\varphi}{\partial x}\right\vert + 
\left\vert\frac{\partial\varphi}{\partial y}\right\vert + 
\left\vert\frac{\partial\varphi}{\partial z}\right\vert .
\eeq
This inequality has been obtained by using the following 
simple identity

\begin{displaymath}
(p + q + s)^2 = p^2 + q^2 + s^2 + 2pq + 2s(p + q) \geq p^2 + q^2 + s^2 ,
\end{displaymath}
for $p, q, s \in \R_{+}.$ This implies 

\begin{displaymath}
p + q + s \geq (p^2 + q^2 + s^2)^{1/2}.
\end{displaymath}
To arrive at (\ref{gradient_norm_bound}), set 
$p := \left\vert \frac{\partial\varphi}{\partial x} \right\vert,$
$q := \left\vert \frac{\partial\varphi}{\partial y} \right\vert,$
and lastly $s := \left\vert \frac{\partial\varphi}{\partial z} \right\vert.$
Now by the definition of partial derivative in the componentwise sense,

\bea
\norm{\nabla\varphi}_2  \leq \lim_{x \rightarrow \tx_{o}}\frac{\vert\varphi(x, y, z) - \varphi(\tx_{o}, \ty, \tz)\vert}{\vert x - \tx_{o} \vert}
& + & \lim_{y \rightarrow \ty_{o}}\frac{\vert\varphi(x, y, z) - \varphi(\tx, \ty_{o}, \tz)\vert}{\vert y - \ty_{o} \vert} + 
\nonumber\\
& + & \lim_{z \rightarrow \tz_{o}}\frac{\vert\varphi(x, y, z) - \varphi(\tx, \ty, \tz_{o})\vert}{\vert z - \tz_{o} \vert}
\nonumber
\eea
Gradient of the functional $\varphi(\bx)$ over the
compact domain is valid for any $\tilde{\bx}_{o} \in \Omega.$ 
Therefore, we continue with our proof in the unified form. 
First observe that by the Lebesgue dominated convergence 
theorem,

\bea
\label{integral_bound1}
\int_{\Omega}\norm{\nabla\varphi(\tilde{\bx}_{o})}_2 d\tilde{\bx}_{o}
& \leq & \int_{\Omega} \lim_{\bx \rightarrow \tilde{\bx}_{o}} 
\left\{\frac{\vert\varphi(\bx) - \varphi(\tilde{\bx}_{o})\vert}
{\Vert \bx - \tilde{\bx}_{o} \Vert_{2}} \right\} d\tilde{\bx}_{o} 
\nonumber\\
& = &  \lim_{\bx \rightarrow \tilde{\bx}_{o}} \int_{\Omega} 
\left\{\frac{\vert\varphi(\bx) - \varphi(\tilde{\bx}_{o})\vert}
{\Vert \bx - \tilde{\bx}_{o} \Vert_{2}}\right\} d\tilde{\bx}_{o} .
\eea
Since $\varphi \in \cC^{1}(\Omega),$ then H\"{o}lder 
continuity given by (\ref{holder_cont}) is satisfied
for $\gamma = 1,$

\begin{displaymath}
\vert \varphi(\bx) - \varphi(\tilde{\bx}_{o}) \vert
\leq \kappa \Vert \bx - \tilde{\bx}_{o} \Vert_{2}
\end{displaymath}
which is Lipschitz continuity. Then (\ref{integral_bound1})
reads

\beq
\int_{\Omega}\norm{\nabla\varphi(\tilde{\bx}_{o})}_2 d\tilde{\bx}_{o} \leq
 \lim_{\bx \rightarrow \tilde{\bx}_{o}} \int_{\Omega} \kappa d\tilde{\bx}_{o} = \kappa\vert \Omega \vert .
\eeq

\end{proof}

We formulate the last formulatin for this section 
which is an immediate consequence of this theorem.

\begin{cor}
\label{euclidean_grad_squared}

Under the same conditions of Theorem \ref{holder_embedded_into_TV}, 
then,

\beq
\int_{\Omega} \norm{\nabla \varphi}_2^2 dx \leq \kappa^{2}\vert \Omega \vert^{2} .
\eeq
\end{cor}

\begin{proof}
Again, by the definition of Euclidean norm in
Section \ref{notations} by (\ref{magnitude_of_the_grad}),

\begin{displaymath}
\norm{\nabla \varphi}_2^2 = \left\vert \frac{\partial \varphi}{\partial x} \right\vert^{2} +
\left\vert \frac{\partial \varphi}{\partial y} \right\vert^{2} +
\left\vert \frac{\partial \varphi}{\partial z} \right\vert^{2} .
\end{displaymath}
Analogous to the proof of Theorem \ref{holder_embedded_into_TV},

\bea
\label{integral_bound}
\int_{\Omega}\norm{\nabla\varphi(\tilde{\bx}_{o})}_2^2 d\tilde{\bx}_{o}
& \leq & \int_{\Omega} \lim_{\bx \rightarrow \tilde{\bx}_{o}} 
\left\{\frac{\vert\varphi(\bx) - \varphi(\tilde{\bx}_{o})\vert}
{\Vert \bx - \tilde{\bx}_{o} \Vert_{2}} \right\}^2 d\tilde{\bx}_{o} 
\nonumber\\
& = &  \lim_{\bx \rightarrow \tilde{\bx}_{o}} \int_{\Omega} 
\left\{\frac{\vert\varphi(\bx) - \varphi(\tilde{\bx}_{o})\vert}
{\Vert \bx - \tilde{\bx}_{o} \Vert_{2}}\right\}^2 d\tilde{\bx}_{o} 
\nonumber\\
& \leq & \kappa^2 \vert \Omega \vert^2
\eea
since $\varphi(\bx) \in \cC^{1}(\Omega).$
\end{proof}


\section{Smoothed-TV Regularization Is an Admissible Regularization Strategy 
With the H\"{o}lder Continuity}
\label{tv_application}

We will define such a regularization parameter which will simultaneously 
enable us to prove the convergence of the smoothed-TV
regularization and to estimate the dicrepancy
$\norm{\cT\varphi- f^{\delta}}$ for the corresponding
regularization strategy, \cite{ChanGolubMulet99}.
Unlike the available literature, 
\cite{AcarVogel94, BachmayrBurger09, BardsleyLuttman09, ChambolleLions97,
ChanChen06, ChanGolubMulet99, DobsonScherzer96, 
DobsonVogel97, VogelOman96},
we define discrepancy principle for the smoothed-TV regularization
under a particular rule for the choice of regularization parameter.
Furthermore, still with the same regularization parameter,
we manage to show that smoothed-TV regularization is an 
admissible regularization strategy with H\"{o}lder continuity.
Throughout this section, the fact that our targeted
solution function is H\"{o}lder continuous will be to our benefit
to be able provide an implementable regularization parameter 
for copmuterized environment.
Hereafter, the component $\bx$ is replaced by $x$ only 
for the sake of simplicity. 

To be able to show the convergence of 
$\norm{\varphi_{\alpha(\delta)} - \varphi^{\dagger}},$
we will refer to Bregman divergence. In Proposition 
\ref{proposition_q-convexity}, we have demonstrated
the relation between strong convexity and $2-$convexity.
Convexity of the smoothed total variation penalizer
has been established in \cite[Theorem 2.4]{AcarVogel94}.
We will ensure the strong convexity of the same penalizer
in the following formulation.

\begin{theorem}
\label{smoothed_TV_strongly_convex}

For any $\beta > 0,$ the functional 
$J(\varphi) := \int_{\Omega} \sqrt{\norm{\nabla\varphi}_2^2 + \beta} dx$
is strongly convex.
\end{theorem}

\begin{proof}

It suffices to prove that $\nabla^2 J(\varphi) \geq 0.$
To avoid confusion in the calculations, we will make an 
assignment $g(p) = \sqrt{\vert p \vert^2 + \beta}$
where $p =  \nabla\varphi .$ 
According to Leibniz integral rule, calculating 
$\nabla^2 J(\varphi) \geq 0$ and $g^{''}(p)$
are equivalent to each other. Then

\begin{displaymath}
g^{'}(p) = \frac{p}{\sqrt{p^2 + \beta}},
\end{displaymath}
and likewise

\begin{displaymath}
g^{''}(p) = \frac{\beta}{(p^2 + \beta)^{3/2}} .
\end{displaymath}
Obviously, $g^{''}(p) > 0$ for any $\beta > 0.$

%

\end{proof}

\begin{theorem}
\label{bound_for_successive_TV}
Over the compact domain $\Omega \subset \R^{3},$
assume that $u, v \in \cC^1(\Omega) \bigcap \cC^{0,\gamma}(\Omega).$
Then there exists a dynamical positive real-valued functional 
$\cK(u) : \cC^1(\Omega) \rightarrow \R_{+}$ depending 
on $u$ such that

\bea
J(u) - J(v) \leq 2\kappa^2 \vert \Omega \vert^2 \cK(u) \norm{\nabla(u-v)}_2 ,
\nonumber
\eea
for $J(\cdot)$ defined by 
$J(\cdot) = \int_{\Omega} \sqrt{\norm{\nabla(\cdot)}_2^2 + \beta} dx$
and where $\kappa$ satisfies (\ref{holder_cont}) for $\gamma  =1.$
\end{theorem}
%

\begin{proof}
By the definition of $J(\cdot) := \int_{\Omega} \sqrt{\norm{ \nabla \cdot }_2^2 + \beta} dx,$

\bea
J(u) - J(v) & = & \int_{\Omega}\frac{\norm{ \nabla u }_2^2 - \norm{ \nabla v }_2^2}{(\norm{ \nabla u }_2^2 + \beta)^{1/2} + (\norm{ \nabla v }_2^2 + \beta)^{1/2}} dx
\nonumber\\
& \leq & \int_{\Omega}\frac{\norm{ \nabla u }_2^2 - \norm{ \nabla v }_2^2}{(\norm{ \nabla v }_2^2 + \beta)^{1/2} } dx.
\eea
Now choose $\cU = \min_{x \in \Omega} \{ \norm{ \nabla u }_2^2 \}$ to have,

\bea
J(u) - J(v) & \leq & \frac{1}{(\cU + \beta)^{1/2}} \int_{\Omega} (\norm{ \nabla u }_2^2 - \norm{ \nabla v }_2^2) dx
\nonumber\\
& = & \frac{1}{(\cU + \beta)^{1/2}} \int_{\Omega}(\norm{ \nabla u }_2 - \norm{ \nabla v }_2)(\norm{ \nabla u }_2 + \norm{ \nabla v }_2) dx
\nonumber\\
& \leq & \frac{1}{(\cU + \beta)^{1/2}} \int_{\Omega} (\norm{ \nabla u - \nabla v }_2 ) (\norm{ \nabla u }_2 + \norm{ \nabla v }_2) dx .
\nonumber
\eea
Apply H\"{o}lder inequality to have,

\bea
J(u) - J(v) & \leq & \frac{1}{(\cU + \beta)^{1/2}}
\left( \int_{\Omega} \norm{ \nabla u - \nabla v }_2^2 dx \right)^{1/2}  
\left( \int_{\Omega} (\norm{ \nabla u }_2 + \norm{ \nabla v }_2)^2 dx \right)^{1/2}
\nonumber\\ 
& \leq & \frac{2}{(\cU + \beta)^{1/2}} \left( \int_{\Omega} \norm{ \nabla u - \nabla v }_2^2 dx \right)^{1/2}
\left( \int_{\Omega} \norm{ \nabla u }_2^2 + \norm{ \nabla v }_2^2 dx \right)^{1/2} ,
\nonumber
\eea
since $2 a b \leq a^2 + b^2$ for any $a,b \in \R_{+}.$ 
By Corollary \ref{euclidean_grad_squared},
we have already obtained the upper bound for 
the second integral on the right hand side. Then,

\bea
J(u) - J(v) & \leq & \frac{2\kappa^2 \vert \Omega \vert^2}{(\cU + \beta)^{1/2}} 
\left( \int_{\Omega} \norm{ \nabla u - \nabla v }_2^2 dx \right)^{1/2}
\nonumber\\
& = & \frac{2\kappa^2 \vert \Omega \vert^2}{(\cU + \beta)^{1/2}} \norm{\nabla(u-v)}_2 .
\nonumber
\eea
Hence, the positive real valued functional is defined by,

\begin{displaymath}
\cK(u) := \frac{1}{(\cU + \beta)^{1/2}} .
\end{displaymath}

\end{proof}

\noindent An immediate consequence that we make use of
$\cC^{1+}(\Omega)$ function space is formulated below.

\begin{cor}
\label{lipschitz_bound_for_successive_TV}
Over the compact domain $\Omega \subset \R^{3},$
assume that $u, v \in \cC^{1+}(\Omega) \bigcap \cC^{0,\gamma}(\Omega).$
Then with the same functional 
$\cK(u) : \cC^{1+}(\Omega) \rightarrow \R_{+}$ as appears
in Theorem \ref{bound_for_successive_TV}, it is hold that

\bea
J(u) - J(v) \leq 2\kappa^2 L_{u,v} \vert \Omega \vert^2 \cK(u) ,
\nonumber
\eea
for $J(\cdot)$ defined by $J(\cdot) = \int_{\Omega} \sqrt{\norm{\nabla(\cdot)}_2^2 + \beta} dx,$
where $\kappa$ satisfies (\ref{holder_cont}) for $\gamma  = 1.$
\end{cor}

\begin{proof}

Since $u, v \in \cC^{1+}(\Omega),$ there there exists constant satisfying 
$\norm{\nabla(u-v)}_2 \leq L_{u,v}\norm{u - v}_2.$ 
Then it follows from the above calculations in the proof of
Theorem \ref{bound_for_successive_TV},

\bea
J(u) - J(v) & \leq & \frac{2\kappa^2 \vert \Omega \vert^2}{(\cU + \beta)^{1/2}} \norm{\nabla(u-v)}_2
\nonumber\\
& \leq & \frac{2\kappa^2 \vert \Omega \vert^2}{(\cU + \beta)^{1/2}} L_{u,v}\norm{u - v}_2 .
\nonumber
\eea
\end{proof}


\subsection{Discrepancy principle for the smoothed TV regularizer}
\label{discrepancy_for_TV}

We are able to evaluate the fixed coefficient 
$\tau$ in the discrepancy principle
$\norm{\cT\varphi_{\alpha(\delta)} - f^{\delta}} \leq \tau\delta$ 
for the smoothed TV penalty $J(\cdot)$ in the problem
(\ref{problem}). To do so, we need to assume that the target function
to be class of $\cC^{1+}(\Omega).$

Moreover, in order for a precise upper bound for 
$\norm{\cT\varphi_{\alpha(\delta)} - f^{\delta}},$
we will need to focus  on our specified penalty
$J(\varphi) := \int_{\Omega} \sqrt{\norm{\nabla \varphi}_2^2 + \beta} dx.$ 
The regularized solution $\varphi_{\alpha(\delta)}$
to the problem (\ref{problem}) is the minimum of $F(\varphi)$
for all $\varphi \in \cD(\cT).$ Which is in other words,

\bea
\varphi_{\alpha(\delta)} \in \argmin_{\varphi} 
\left\{ F_{\alpha}(\varphi) = \frac{1}{2} \norm{\cT\varphi - f^{\delta}}_{\cL^2}^2 + \alpha J(\varphi) \right\} .
\nonumber
\eea
Then

\beq
\label{minimizer_true_1}
\frac{1}{2}\norm{\cT \varphi_{\alpha(\delta)} - f^{\delta}}_{\cL^2}^2 + \alpha J(\varphi_{\alpha(\delta)})
\leq \frac{1}{2}\norm{\cT \varphi^{\dagger} - f^{\delta}}_{\cL^2}^2 + \alpha J(\varphi^{\dagger}).
\eeq
Since the true data $f^{\dagger}$ satisfying the operator 
equation $\cT\varphi^{\dagger} = f^{\dagger}$ lies in some
$\delta-$ball $\cB_{\delta}(f^{\delta}),$ {\em i.e.}
$\norm{f^{\dagger} - f^{\delta}} \leq \delta,$
then (\ref{minimizer_true_1}) reads,

\beq
\label{minimizer_true_2}
\frac{1}{2}\norm{\cT \varphi_{\alpha(\delta)} - f^{\delta}}_{\cL^2}^2 \leq 
\frac{1}{2}\delta^{2} + \alpha (J(\varphi^{\dagger}) - J(\varphi_{\alpha(\delta)})) .
\eeq
Further development of this estimation will be done by means of
Theorem \ref{holder_embedded_into_TV} as formulated below.

\begin{theorem}[Discrepancy principle for the smoothed-TV regularization]
\label{discrepancy_for_smoothed_TV}
Over the compact domain $\Omega \subset \R^{3},$ denote by 
$\varphi_{\alpha(\delta)}, \varphi^{\dagger} \in \cC^{1+}(\Omega) \bigcap \cC^{0,\gamma}(\Omega)$ 
the regularized and the true solutions to the problem 
(\ref{problem}) respectively.
If the regularization parameter $\alpha(\delta)$ 
is chosen according to the rule of,

\beq
\label{regpar_tv_discrepancy1}
\alpha(\delta) \leq \delta^2 \frac{\left(\cK(\varphi_{\alpha(\delta)})\right)^{-1}}{2 \kappa^2 \vert\Omega\vert^2} ,
\eeq
for $\kappa$ satisfying (\ref{holder_cont}) with $\gamma = 1,$ and

\begin{displaymath}
\cK(\varphi_{\alpha(\delta)}) := \frac{1}{(\cU + \beta)^{1/2}}
\end{displaymath}
where

\begin{displaymath}
\cU = \min_{x \in \Omega} \{ \norm{ \nabla \varphi_{\alpha(\delta)} }_2^2 \},
\end{displaymath} 
then the discrepancy $\norm{\cT \varphi_{\alpha(\delta)} - f^{\delta}}_{\cL^2}$
for the smoothed-TV regularization is estimated by,

\beq
\label{discrepancy_for_smoothed_TV_est1}
\norm{\cT \varphi_{\alpha(\delta)} - f^{\delta}}_{\cL^2}  \leq 
\delta\sqrt{1 + \norm{\nabla(\varphi_{\alpha(\delta)} - \varphi^{\dagger})}_2} .
\eeq
Furthermore, if the regularization parameter fulfils

\beq
\label{regpar_tv_discrepancy2}
\alpha(\delta) \leq \delta^2 \frac{\left(\cK(\varphi_{\alpha(\delta)})\right)^{-1}}
{2 \kappa^2 \vert\Omega\vert^2 L_{\varphi_{\alpha(\delta)} , \varphi^{\dagger}}},
\eeq
where $L_{\varphi_{\alpha(\delta)} , \varphi^{\dagger}}$ 
is an appropriate Lipschitz costant then,

\beq
\label{discrepancy_for_smoothed_TV_est2}
\norm{\cT \varphi_{\alpha(\delta)} - f^{\delta}}_{\cL^2}  \leq 
\delta\sqrt{1 + \norm{\varphi_{\alpha(\delta)} - \varphi^{\dagger}}_2} .
\eeq
\end{theorem}

\begin{proof}

From  the calculations in (\ref{minimizer_true_2}) and the quick adaptation of 
Theorem \ref{bound_for_successive_TV}, it is firstly obtained
that

\bea
\label{discrepency_smoothed_TV1}
\norm{\cT \varphi_{\alpha(\delta)} - f^{\delta}}_{\cL^2}^2 & \leq &
\delta^{2} + 2 \alpha (J(\varphi_{\alpha(\delta)}) - J(\varphi^{\dagger}))
\nonumber\\
& \leq & \delta^{2} + 2 \alpha \kappa^{2}\vert\Omega\vert^2 \cK(\varphi_{\alpha(\delta)})
\norm{\nabla(\varphi_{\alpha(\delta)} - \varphi^{\dagger})}_2 .
\eea
Then, with the given rule of the regularization parameter $\alpha$
in (\ref{regpar_tv_discrepancy1}), 

\bea
\norm{\cT \varphi_{\alpha(\delta)} - f^{\delta}}_{\cL^2}^2 \leq 
\delta^2 + \delta^2\norm{\nabla(\varphi_{\alpha(\delta)} - \varphi^{\dagger})}_2 ,
\nonumber
\eea
which is the first result. It is not difficult to obtain
the second part of the theorem. Analogous to 
Corollary \ref{lipschitz_bound_for_successive_TV}, observe that
there exists $L_{\varphi_{\alpha(\delta)} , \varphi^{\dagger}}$
such that

\bea
J(\varphi_{\alpha(\delta)}) - J(\varphi^{\dagger}) & \leq & 
\kappa^{2}\vert\Omega\vert^2 \cK(\varphi_{\alpha(\delta)})
\norm{\nabla(\varphi_{\alpha(\delta)} - \varphi^{\dagger})}_2
\nonumber\\
& \leq & \kappa^{2}\vert\Omega\vert^2 \cK(\varphi_{\alpha(\delta)})
L_{\varphi_{\alpha(\delta)} , \varphi^{\dagger}}
\norm{\varphi_{\alpha(\delta)} - \varphi^{\dagger}}_2 
\noindent
\eea
Then, from (\ref{discrepency_smoothed_TV1}),

\bea
\norm{\cT \varphi_{\alpha(\delta)} - f^{\delta}}_{\cL^2}^2 \leq
\delta^{2} + 2 \alpha \kappa^{2}\vert\Omega\vert^2 \cK(\varphi_{\alpha(\delta)})
L_{\varphi_{\alpha(\delta)} , \varphi^{\dagger}} 
\norm{\varphi_{\alpha(\delta)} - \varphi^{\dagger}}_2 .
\eea
Hence, with the given rule for choice of the regularization parameter $\alpha$
in (\ref{regpar_tv_discrepancy2}), the second desired result yields.
\end{proof}

%

Uniform continuity of the smoothed-TV regularization will come from
formulating another useful Bregman divergence which will lead us to the 
ultimate result of this work.
Before proceeding, it must be noted that the discrepancy principle for the smoothed
TV is yet to be completed in (\ref{discrepancy_for_smoothed_TV_est2})
which is necessary in order for fulfilling the condition in (\ref{discrepancy_pr_definition}).
The completion will follow after quantifying the rate for 
$\norm{\varphi_{\alpha(\delta)} - \varphi^{\dagger}}_2.$

\begin{theorem}
\label{thrm_tv_convergence}
Under the assumptions of Theorem \ref{discrepancy_for_smoothed_TV},
if the regularization parameter, for $\delta \in (0,1),$ 
satisfies

\bea
\label{regpar_tv_discrepancy3}
\alpha(\delta) \leq \delta \frac{\left(\cK(\varphi_{\alpha(\delta)})\right)^{-1}}
{2 \kappa^2 \vert\Omega\vert^2 L_{\varphi_{\alpha(\delta)} , \varphi^{\dagger}}},
\nonumber
\eea
which is analogous to (\ref{regpar_tv_discrepancy2})
in Theorem \ref{discrepancy_for_smoothed_TV}, then

\begin{displaymath}
\norm{\varphi_{\alpha(\delta)} - \varphi^{\dagger}}_2 \rightarrow 0, \mbox{ as $\alpha(\delta) \rightarrow 0$ whilst $\delta \rightarrow 0$ }. 
\end{displaymath}

\end{theorem}

\begin{proof}
Since we will prove the assertion depending on the choice
of the regularization parameter, then we formulate
another Bregman divergence associated with the functional
$\alpha J(\cdot).$ Moreover, it is clear that
first order optimality conditions in (\ref{optimality_1})
must also be hold for the true solution $\varphi^{\dagger}.$
Then,

\bea
D_{\alpha J}(\varphi_{\alpha(\delta)} , \varphi^{\dagger})
& = & \alpha J(\varphi_{\alpha(\delta)}) - \alpha J(\varphi^{\dagger})
- \langle \alpha \nabla J(\varphi^{\dagger}), \varphi_{\alpha(\delta)} - \varphi^{\dagger} \rangle
\nonumber\\
& = & \alpha \left( J(\varphi_{\alpha(\delta)}) - J(\varphi^{\dagger}) \right) - 
\langle \cT^{\ast}(f^{\delta} - \cT\varphi^{\dagger}) , \varphi_{\alpha(\delta)} - \varphi^{\dagger} \rangle .
\eea
Again, from  the calculations in (\ref{minimizer_true_2}) and by
Theorem \ref{bound_for_successive_TV}, there exists Lipschitz
constant $L_{\varphi_{\alpha(\delta)} , \varphi^{\dagger}}$
such that

\bea
D_{\alpha J}(\varphi_{\alpha(\delta)} , \varphi^{\dagger})
\leq \alpha \kappa^{2}\vert\Omega\vert^2 \cK(\varphi_{\alpha(\delta)})
L_{\varphi_{\alpha(\delta)} , \varphi^{\dagger}}
\norm{\varphi_{\alpha(\delta)} - \varphi^{\dagger}}_2 + 
\delta \norm{\cT} \norm{\varphi_{\alpha(\delta)} - \varphi^{\dagger}}_2 .
\nonumber
\eea
By the given choice of regularization parameter rule (\ref{regpar_tv_discrepancy3}),
it is concluded that

\bea
D_{\alpha J}(\varphi_{\alpha(\delta)} , \varphi^{\dagger}) \leq \delta \norm{\varphi_{\alpha(\delta)} - \varphi^{\dagger}}_2
+ \delta \norm{\cT^{\ast}} \norm{\varphi_{\alpha(\delta)} - \varphi^{\dagger}}_2 .
\eea
Hence, by $2-$convexity in (\ref{q_convexity}) and strong convexity of
the smoothed TV penalty (see Thrm \ref{smoothed_TV_strongly_convex}),

\beq
\norm{\varphi_{\alpha(\delta)} - \varphi^{\dagger}}_2 \leq 
\delta \left(\frac{1}{2} + \norm{\cT^{\ast}} \right) \rightarrow 0, \mbox{ as $\alpha(\delta) \rightarrow 0$ whilst $\delta \rightarrow 0$ } .
\eeq

\end{proof}

\begin{cor}[Final discrepancy estimation for the smoothed-TV regularization]
\label{final_discrepancy}
Under the same assumptions of Theorem \ref{thrm_tv_convergence},
if the regularization parameter is chosen with respect to the rule

\bea
\label{regpar_tv_discrepancy4}
\alpha(\delta) \leq \delta \frac{\left(\cK(\varphi_{\alpha(\delta)})\right)^{-1}}
{2 \kappa^2 \vert\Omega\vert^2 L_{\varphi_{\alpha(\delta)} , \varphi^{\dagger}}},
\eea
for $\delta \in (0,1)$ sufficiently small,
then by (\ref{regpar_tv_discrepancy2}),

\bea
\norm{\cT \varphi_{\alpha(\delta)} - f^{\delta}}_{\cL^2} & \leq &
\delta \left(\frac{3}{2} + \norm{\cT^{\ast}} \right)^{1/2} .
\eea

\end{cor}

\begin{proof}
In the proof of Theorem \ref{thrm_tv_convergence},
we have observed

\bea
\norm{\cT \varphi_{\alpha(\delta)} - f^{\delta}}_{\cL^2}^2 \leq
\delta^{2} + 2 \alpha \kappa^{2}\vert\Omega\vert^2 \cK(\varphi_{\alpha(\delta)})
L_{\varphi_{\alpha(\delta)} , \varphi^{\dagger}} \norm{\varphi_{\alpha(\delta)} - \varphi^{\dagger}}_2 .
\eea
After plugging (\ref{regpar_tv_discrepancy4}) in,

\bea
\norm{\cT \varphi_{\alpha(\delta)} - f^{\delta}}_{\cL^2}^2 \leq
\delta^{2} + \delta \norm{\varphi_{\alpha(\delta)} - \varphi^{\dagger}}_2 .
\eea
Recall the rate of the convergence for $\norm{\varphi_{\alpha(\delta)} - \varphi^{\dagger}}_2$
when the choice of regularization parameter fulfils the condition 
(\ref{regpar_tv_discrepancy4}), which is

\bea
\norm{\cT \varphi_{\alpha(\delta)} - f^{\delta}}_{\cL^2}^2 \leq
\delta^{2} + \delta^{2} \left(\frac{1}{2} + \norm{\cT^{\ast}} \right) = 
\delta^2 \left(\frac{3}{2} + \norm{\cT^{\ast}} \right).
\nonumber
\eea
This yields the result after taking the square root of both sides.

\end{proof}

\section*{Conclusion and Further Discussion}

H\"{o}lder continuous functions in the application arises in the field of
scattering theory, see for details \cite[Section 8.2]{ColtonKress13}
and a recent work \cite[Lemma 2.2]{HuSaloVesalainen15}.
In this work, we have explored the regularity properties of 
the smoothed-TV regularization for such functions.
The scientific reason why smoothed-TV regularization has been chosen
for such a study has been established in Theorem \ref{Holder_bounded_by_TV}.

In Theorem \ref{holder_embedded_into_TV}, we have proved that
$TV$ of a function $\varphi$ can be bounded by its Lipschitz
constant $\kappa$ which is the case of $\gamma = 1.$
However, it is still an open question to be able to show that

\begin{displaymath}
TV(\varphi) \leq C(\Omega) [\varphi]_{\cC^{0,\gamma}(\Omega)}
\end{displaymath}
where $C(\Omega)$ is a constant depending on the compact
domain $\Omega \subset \R^{3}$ and $\gamma = 1/4$ due to
Morrey's inequality, see Theorem \ref{Morrey_ineq}.
Then a new compact embedding theorem between the spaces $BV(\Omega)$
and $\cC^{0,\gamma}(\Omega)$ can be established.
On the other hand, compact embedding amongst the H\"{o}lder spaces with the different
orders has already been proven, \cite[Theorem 3.2]{ColtonKress13}.

Speaking about proving that smoothed TV regularization is another
admissible regularization strategy, we have intentionally taken
into account that the forward operator is compact. The reason
behind that can be explained as follows; Application and analysis of 
the method has been widely carried
out in the communities of inverse problem and optimization, 
\cite{AcarVogel94, BachmayrBurger09, BardsleyLuttman09, ChambolleLions97,
ChanChen06, ChanGolubMulet99, DobsonScherzer96, 
DobsonVogel97, VogelOman96}. It is well-known that the efficient result
of TV (or smoothed TV) regularization usually comes from image
processing where the compact operator is mostly considered
to be identity operator, {\em i.e.,} $\cT = \cI.$ Lagged diffusivity fixed
point iteration is the easiest algorithm in order to approximate
the solution for the problem (\ref{problem}), 
\cite{ChanGolubMulet99, Vogel02, VogelOman96}. The convergence
of this algorithm has been shown only for the case of $\cT = \cI,$
\cite{Aujol09, ChanGolubMulet99}. Following the same steps
in the regarding works, we also define the following continuous 
nonlinear transformation

\beq
\cP(\varphi_{\varphi_{\alpha(\delta)}}) := \left(-\alpha(\delta) \nabla^{\ast} \cdot \left( \frac{\nabla}
{( \beta + \vert \nabla \varphi_{\alpha(\delta)} \vert^2)^{1/2}} \right) + \cT^{\ast}\cT \right) .
\eeq
According to the regarding works, the algorithm is convergent in the condition of 
$\lambda_{\min}(\cP(\varphi_{\varphi_{\alpha(\delta)}})) \geq \sigma(\cT^{\ast}\cT) \geq 1 .$
Obviously, this can not hold for us since our forward operator
$\cT$ is compact. A tomographic application of total variation
regularization with some compact foward operator has been recently studied
in \cite{HauptmannSiltanen14}.

\bigskip
\section*{References}


 \bibliographystyle{alpha}

\end{document}